\documentclass[12pt,reqno]{article}

\usepackage[usenames]{color}
\usepackage{amssymb}
\usepackage{graphicx}
\usepackage{amscd}

\usepackage{amsthm}
\newtheorem{theorem}{Theorem}

\newtheorem{proposition}[theorem]{Proposition}
\newtheorem{corollary}[theorem]{Corollary}

\theoremstyle{definition}

\usepackage[colorlinks=true,
linkcolor=webgreen, filecolor=webbrown,
citecolor=webgreen]{hyperref}

\definecolor{webgreen}{rgb}{0,.5,0}
\definecolor{webbrown}{rgb}{.6,0,0}

\usepackage{color}

\usepackage{float}

\usepackage{graphics,amsmath,amssymb}
\usepackage{amsfonts}
\usepackage{latexsym}
\usepackage{epsf}

\newcommand{\seqnum}[1]{\href{http://www.research.att.com/cgi-bin/access.cgi/as/~njas/sequences/eisA.cgi?Anum=#1}{\underline{#1}}}

\begin{document}

\begin{center}
\vskip 1cm{\LARGE\bf On the restricted Chebyshev-Boubaker polynomials} \vskip 1cm \large
Paul Barry\\
School of Science\\
Waterford Institute of Technology\\
Ireland\\
\href{mailto:pbarry@wit.ie}{\tt pbarry@wit.ie}
\end{center}
\vskip .2 in

\begin{abstract} Using the language of Riordan arrays, we study a one-parameter family of orthogonal polynomials that we call the restricted Chebyshev-Boubaker polynomials. We characterize these polynomials in terms of the three term recurrences that they satisfy, and we study certain central sequences defined by their coefficient arrays. We give an integral representation for their moments, and we show that the Hankel transforms of these moments have a simple form. We show that the (sequence) Hankel transform of the row sums of the corresponding moment matrix is defined by a family of polynomials closely related to the Chebyshev polynomials of the second kind, and that these row sums are in fact the moments of another family of orthogonal polynomials. \end{abstract}

\section{Introduction}

In this paper, we shall explore a number of polynomial sequences arising from a study of the one-parameter family of orthogonal polynomials defined by the Riordan array
$$\left(\frac{1+rx^2}{1+x^2}, \frac{x}{1+x^2}\right).$$ Elements of this family may be regarded as generalized scaled Chebyshev polynomials of the second kind.
We use ideas from the theory of the Riordan group of lower-triangular matrices \cite{SGWW} to help us derive our results. General information on orthogonal polynomials as used in this note may be found in, for instance, \cite{Chihara, Gautschi, Szego}. We will find it useful to give continued fraction expressions for  generating functions encountered in this note \cite{Wall}, particularly when we want to characterize Hankel transforms \cite{Kratt} of sequences. In the next section, we give a brief overview of Riordan arrays, and their links to orthogonal polynomials. In the following section, we concentrate on the coefficient array $(P_{n,k})$ of the orthogonal polynomials, and in particular we study properties of the sequences $P_{2n,n}$ and $P_{2n,n+1}$. In this context, we note that the Boubaker polynomials (the case $r=3$ of the family parameter) have the special property $P_{2n,n}=0^n$.
In the next and final section, we look at the moment matrix (the inverse of the coefficient array), its first column, which is the moment sequence of the family of orthogonal polynomials under study, and the row sums of the moment matrix. We give integral representations of these sequences, and we characterize their Hankel transforms. We find that this family of orthogonal polynomials has an interesting property: the Hankel transform of the row sums of the moment matrix is expressible in terms of suitably scaled and shifted Chebyshev polynomials of the second kind.

\section{Riordan arrays and orthogonal polynomials}
The group of Riordan arrays $\mathcal{R}$ \cite{SGWW}  was first introduced by Shapiro, Getu, Woan, and Woodson in the early $1990$'s. Since then, they have been extensively studied and applied in a number of different fields.
At its simplest, a Riordan array is formally defined by a pair of power series, say $g(x)$ and $f(x)$, where $g(0)=1$ and $f(x)=x+a_2 x^2+a_3x^2+\ldots$, with integer coefficients (such Riordan arrays are called ``proper'' Riordan arrays). The pair $(g, f)$ is then associated to the lower-triangular invertible matrix whose $(n,k)$-th element $T_{n,k}$ is given by
$$T_{n,k}=[x^n] g(x)f(x)^k.$$

In this paper, we shall define a one-parameter family of orthogonal polynomials using Riordan arrays, and we shall investigate a number of aspects of these orthogonal polynomials, notably the central sequences defined by their coefficient array, and the Hankel transform \cite{Layman} of their moment sequence. We recall that for a sequence $a_n$ we define its \emph{Hankel transform} to be the sequence of determinants $h_n=|a_{i+j}|_{0 \le i,j \le n}$.

All the power series and matrices that we shall look at are assumed to have integer coefficients. Thus power series are elements of $\mathbb{Z}[[x]]$. The generating function $1$ generates the sequence that we denote by $0^n$, which begins $1,0,0,0,\ldots$. All matrices are assumed to begin at the $(0,0)$ position, and to extend infinitely to the right and downwards. Thus matrices in this article are elements of $\mathbb{Z}^{\mathbb{N}_0 \times \mathbb{N}_0}$. When examples are given, an obvious truncation is applied.

The Fundamental Theorem of Riordan arrays \cite{Survey} says that the action of a Riordan array on a power series, namely
$$ (g(x), f(x))\cdot a(x)= g(x)a(f(x)),$$ is realised in matrix form by
$$\left(T_{n,k}\right)\left( \begin{array}{c}a_0\\a_1\\a_2\\a_3\\ \vdots\end{array}\right)
=\left( \begin{array}{c}b_0\\b_1\\b_2\\b_3\\ \vdots\end{array}\right),$$ where
the power series $a(x)$ expands to give the sequence $a_0, a_1, a_2, \ldots$, and the image sequence $b_0, b_1, b_2, \ldots$ has generating function $g(x)a(f(x))$.

An important feature of Riordan arrays is that they have a number of sequence characterizations \cite{Cheon, He}. The simplest of
these
is as follows.
\begin{proposition} \label{Char} \cite[Theorem 2.1, Theorem 2.2]{He} Let $D=[d_{n,k}]$ be an infinite triangular matrix. Then $D$ is a Riordan array if and only if there
exist two sequences $A=[a_0,a_1,a_2,\ldots]$ and $Z=[z_0,z_1,z_2,\ldots]$ with $a_0 \neq 0$, $z_0 \neq 0$ such that
\begin{itemize}
\item $d_{n+1,k+1}=\sum_{j=0}^{\infty} a_j d_{n,k+j}, \quad (k,n=0,1,\ldots)$
\item $d_{n+1,0}=\sum_{j=0}^{\infty} z_j d_{n,j}, \quad (n=0,1,\ldots)$.
\end{itemize}
\end{proposition}
The coefficients $a_0,a_1,a_2,\ldots$ and $z_0,z_1,z_2,\ldots$ are called the $A$-sequence and the $Z$-sequence of the Riordan array
$D=(g(x),f(x))$, respectively.
Letting $A(x)$ be the generating function of the $A$-sequence and $Z(x)$ be the generating function of the $Z$-sequence, we have
\begin{equation}\label{AZ_eq} A(x)=\frac{x}{\bar{f}(x)}, \quad Z(x)=\frac{1}{\bar{f}(x)}\left(1-\frac{1}{g(\bar{f}(x))}\right).\end{equation}
Here, $\bar{f}(x)$ is the series reversion of $f(x)$, defined as the solution $u(x)$ of the equation
$$f(u)=x$$ that satisfies $u(0)=0$.

The inverse of the Riordan array $(g, f)$ is given by
$$(g(x), f(x))^{-1}=\left(\frac{1}{g(\bar{f}(x))}, \bar{f}(x)\right).$$

For a Riordan array $D$, the matrix $P=D^{-1}\cdot \overline{D}$ is called its \emph{production matrix}, where $\overline{D}$ is the matrix $D$ with its top row removed.

The
concept of a \emph{production matrix} \cite{ProdMat_0,
ProdMat}
is a general one, but for this work we find it convenient to
review it in
the context of Riordan arrays. Thus let $P$ be an infinite
matrix. Letting
$\mathbf{r}_0$
be the row vector
$$\mathbf{r}_0=(1,0,0,0,\ldots),$$ we define $\mathbf{r}_i=\mathbf{r}_{i-1}P$, $i \ge 1$.
Stacking these rows leads to another infinite matrix which we
denote by
$A_P$. Then $P$ is said to be the \emph{production matrix} for
$A_P$.

\noindent If we let $$u^T=(1,0,0,0,\ldots,0,\ldots)$$ then we
have $$A_P=\left(\begin{array}{c}
u^T\\u^TP\\u^TP^2\\\vdots\end{array}\right)$$ and
$$\bar{I}A_P=A_PP$$ where $\bar{I}=(\delta_{i+1,j})_{i,j \ge 0}$ (where
$\delta$ is the usual Kronecker symbol):
\begin{displaymath} \bar{I}=\left(\begin{array}{ccccccc} 0 & 1
& 0 & 0 & 0 & 0 & \ldots \\0 & 0 & 1 & 0 & 0 & 0 & \ldots \\
0 & 0 & 0 & 1
& 0 & 0 & \ldots \\ 0 & 0 & 0 & 0 & 1 & 0 & \ldots \\ 0 & 0 &
0
& 0 & 0 & 1 & \ldots \\0 & 0  & 0 & 0 & 0 & 0 &\ldots\\ \vdots
& \vdots &
\vdots & \vdots & \vdots & \vdots &
\ddots\end{array}\right).\end{displaymath}
We have
\begin{equation}P=A_P^{-1}\bar{I}A_P.\end{equation} Writing
$\overline{A_P}=\bar{I}A_P$, we can write this equation as \begin{equation}P=A_P^{-1}\overline{A_P}.\end{equation} Note that $\overline{A_P}$ is $A_P$ with the first row removed.

The production matrix $P$ is sometimes \cite{P_W, Shapiro_bij} called the Stieltjes matrix $S_{A_P}$
associated to $A_P$. Other examples of the use of production matrices can be found in \cite{Arregui}, for instance.

In the context of Riordan
arrays, the production matrix associated to a proper Riordan
array takes on a special form\,:
 \begin{proposition} \label{RProdMat}
\cite[Proposition 3.1]{ProdMat}\label{AZ} Let $P$ be
an infinite production matrix and let $A_P$ be the matrix
induced by $P$. Then $A_P$ is an (ordinary) Riordan matrix if
and only if $P$ is
of the form \begin{displaymath} P=\left(\begin{array}{ccccccc}
\xi_0 & \alpha_0 & 0 & 0 & 0 & 0 & \ldots \\\xi_1 & \alpha_1 &
\alpha_0 & 0 &
0 & 0 & \ldots \\ \xi_2 & \alpha_2 & \alpha_1 & \alpha_0 & 0 &
0 & \ldots \\ \xi_3 & \alpha_3 & \alpha_2 & \alpha_1 &
\alpha_0
& 0 & \ldots
\\ \xi_4 & \alpha_4 & \alpha_3 & \alpha_2 & \alpha_1 &
\alpha_0
& \ldots \\\xi_5 & \alpha_5  & \alpha_4 & \alpha_3 & \alpha_2
&
\alpha_1
&\ldots\\ \vdots & \vdots & \vdots & \vdots & \vdots & \vdots
&
\ddots\end{array}\right),\end{displaymath} where $\xi_0 \neq 0$, $\alpha_0 \neq 0$. Moreover, columns $0$
and $1$ of
the matrix $P$ are the $Z$- and $A$-sequences,
respectively, of the Riordan array $A_P$. \end{proposition}

Where possible, we shall refer to known sequences and triangles by their OEIS numbers \cite{SL1, SL2}. For instance, the Catalan numbers $C_n=\frac{1}{n+1}\binom{2n}{n}$ with g.f. $c(x)=\frac{1-\sqrt{1-4x}}{2x}$ is the sequence \seqnum{A000108}, the Fibonacci numbers are \seqnum{A000045}, and the Motzkin numbers $M_n=\sum_{k=0}^{\lfloor \frac{n}{2} \rfloor}\binom{n}{2k}C_k$  are \seqnum{A001006}.

The binomial matrix $B=\left(\binom{n}{k}\right)$ is \seqnum{A007318}. As a Riordan array, this is given by
$$B= \left(\frac{1}{1-x}, \frac{x}{1-x}\right).$$

Note that in this article all sequences $a_n$ that have $a_0 \ne 0$ are assumed to have $a_0=1$. Likewise for sequences $b_n$ with $b_0=0$ and $b_1 \ne 0$, we assume that $b_1=1$.

The papers \cite{Meixner, CB} explore the use of Riordan arrays to define constant coefficient orthogonal polynomials. Three relevant results \cite{Meixner} are as follows.
\begin{proposition}
The Riordan array $\left(\frac{1-\lambda x - \mu x^2}{1+rx+sx^2},\frac{x}{1+rx+sx^2}\right)$ is the coefficient array of the generalized Chebyshev polynomials of the second kind given by
$$Q_n(x)=(\sqrt{s})^n U_n\left(\frac{x-r}{2\sqrt{s}}\right)-\lambda(\sqrt{s})^{n-1} U_{n-1}\left(\frac{x-r}{2\sqrt{s}}\right)-\mu(\sqrt{s})^{n-2} U_{n-2}\left(\frac{x-r}{2\sqrt{s}}\right), \quad n=0,1,2,\ldots$$
\end{proposition}
\begin{proposition}
The Riordan array $L$ where
$$L^{-1}=\left(\frac{1-\lambda x - \mu x^2}{1+ax+bx^2},
\frac{x}{1+ax+bx^2}\right)^{-1}$$ has production matrix
(Stieltjes matrix)
given by \begin{displaymath}
P=S_L=\left(\begin{array}{ccccccc}
a+\lambda & 1 & 0 & 0 & 0 & 0 & \ldots \\b+\mu & a & 1 & 0 & 0
& 0 & \ldots
\\ 0 & b & a & 1 & 0 & 0 & \ldots \\ 0 & 0 & b & a & 1 & 0 &
\ldots \\ 0 & 0 & 0 & b & a & 1 & \ldots \\0 & 0  & 0 & 0 & b
&
a &\ldots\\
\vdots & \vdots & \vdots & \vdots & \vdots & \vdots &
\ddots\end{array}\right).\end{displaymath}
\end{proposition}
\begin{proposition} If $L=(g(x),f(x))$ is a
Riordan array and $P=S_L$ is tridiagonal of the form
\begin{equation}\label{P_mat}
P=S_L=\left(\begin{array}{ccccccc} a_1 & 1 & 0 & 0 & 0 & 0 &
\ldots \\b_1 & a & 1 & 0 & 0 & 0 & \ldots \\ 0 & b & a & 1 & 0
& 0 & \ldots \\
0 & 0 & b & a & 1 & 0 & \ldots \\ 0 & 0 & 0 & b & a & 1 &
\ldots \\0 & 0  & 0 & 0 & b & a &\ldots\\ \vdots & \vdots &
\vdots & \vdots &
\vdots & \vdots & \ddots\end{array}\right),\end{equation}
then $L^{-1}$ is the coefficient array of the family of
orthogonal polynomials
$p_n(x)$ where $p_0(x)=1$, $p_1(x)=x-a_1$, and
$$p_{n+1}(x)=(x-a)p_n(x)-b_n p_{n-1}(x), \qquad n \ge 2,$$
where $b_n$ is the sequence
$0,b_1,b,b,b,\ldots$. \end{proposition}

\section{Definitions and Properties}
We define the \emph{restricted Chebyshev-Boubaker polynomials $P_n(x;r)$} to be the one-parameter family of orthogonal polynomials $P_n(x;r)$ whose coefficient arrays are defined by the Riordan arrays
$$\left(\frac{1+r x^2}{1+x^2}, \frac{x}{1+x^2}\right), \quad r \in \mathbb{Z}.$$
Here, $r$ is taken to be an integer parameter ($r \in \mathbb{Z}$). When $r=0$, we get the modified Chebyshev polynomials $U_{n}\left(\frac{x}{2}\right)$ \cite{Mason}, while the case $r=3$ coincides with the Boubaker polynomials \cite{Agida, NMR, Boubaker, BPES, Boubaker07, Dada, Labadiah1, Labadiah2, Milanovic, Slama, Zhao}.
Thus the coefficient array of this family of orthogonal polynomials begins
$$P(r)=\left(
\begin{array}{ccccccc}
 1 & 0 & 0 & 0 & 0 & 0 & 0 \\
 0 & 1 & 0 & 0 & 0 & 0 & 0 \\
 r-1 & 0 & 1 & 0 & 0 & 0 & 0 \\
 0 & r-2 & 0 & 1 & 0 & 0 & 0 \\
 1-r & 0 & r-3 & 0 & 1 & 0 & 0 \\
 0 & 3-2 r & 0 & r-4 & 0 & 1 & 0 \\
 r-1 & 0 & 3 (2-r) & 0 & r-5 & 0 & 1 \\
\end{array}
\right),$$  and hence
the polynomials begin
$$P_0(x;r)=1, P_1(x;r)=x, P_2(x;r)=x^2+r-1, P_3(x;r)=x^3+x(r-2),\ldots.$$
We have
$$P_n(x;r)=U_n\left(\frac{x}{2}\right)+r U_{n-2}\left(\frac{x}{2}\right).$$
The inverse of the coefficient matrix begins
$$M(r)=\left(
\begin{array}{ccccccc}
 1 & 0 & 0 & 0 & 0 & 0 & 0 \\
 0 & 1 & 0 & 0 & 0 & 0 & 0 \\
 1-r & 0 & 1 & 0 & 0 & 0 & 0 \\
 0 & 2-r & 0 & 1 & 0 & 0 & 0 \\
 r^2-3 r+2 & 0 & 3-r & 0 & 1 & 0 & 0 \\
 0 & r^2-4 r+5 & 0 & 4-r & 0 & 1 & 0 \\
 -r^3+5 r^2-9 r+5 & 0 & r^2-5 r+9 & 0 & 5-r & 0 & 1 \\
\end{array}
\right),$$ and thus the moment sequence $\mu(n;r)$ for the polynomial family $P_n(x;r)$ begins
$$1, 0, 1 - r, 0, r^2 - 3r + 2, 0, - r^3 + 5r^2 - 9r + 5, 0, r^4 - 7r^3 + 20r^2 - 28r + 14, 0, \ldots.$$

The production matrix of the moment matrix $M(r)$ begins
$$\left(
\begin{array}{ccccccc}
 0 & 1 & 0 & 0 & 0 & 0 & 0 \\
 1-r & 0 & 1 & 0 & 0 & 0 & 0 \\
 0 & 1 & 0 & 1 & 0 & 0 & 0 \\
 0 & 0 & 1 & 0 & 1 & 0 & 0 \\
 0 & 0 & 0 & 1 & 0 & 1 & 0 \\
 0 & 0 & 0 & 0 & 1 & 0 & 1 \\
 0 & 0 & 0 & 0 & 0 & 1 & 0 \\
\end{array}
\right),$$ indicating that the family of orthogonal polynomials $P_n(x;r)$ obeys the three-term recurrence
$$P_n(x;r)=x P_{n-1}(x;r)-P_{n-2}(x;r),$$ with
$$P_0(x;r)=1, P_1(x;r)=x\quad \textrm{and}\quad P_2(x;r)=x^2+r-1.$$

\section{The coefficient matrix $P(r)$}
We calculate the general term $P_{n,k}$ of the coefficient matrix $P(r)$. For this, we use the method of coefficients \cite{MC}.
\begin{proposition} We have
$$P_{n,k}=\binom{\frac{n+k}{2}}{k}(-1)^{\frac{n-k}{2}}\frac{1+(-1)^{n-k}}{2}+r\binom{\frac{n+k-2}{2}}{k}(-1)^{\frac{n-k}{2}}\frac{1+(-1)^{n-k}}{2}+r.0^{n+k}.$$ \end{proposition}
\begin{proof}
We have
\begin{eqnarray*} P_{n,k}&=& \frac{1+rx^2}{1+x} \left(\frac{x}{1+x^2}\right)^k \\
&=&[x^{n-k}](1+rx^2)\left(\frac{1}{1+x^2}\right)^{k+1}\\
&=&[x^{n-k}](1+rx^2) \sum_{j=0}^{\infty} \binom{-(k+1)}{j}x^{2j}\\
&=&[x^{n-k}](1+rx^2 \sum_{j=0}^{\infty} \binom{k+j}{j}(-1)^j x^{2j}\\
&=&[x^{n-k}]\sum_{j=0}^{\infty} \binom{k+j}{j}(-1)^j x^{2j}+r[x^{n-k-2}]\sum_{j=0}^{\infty} \binom{k+j}{j}(-1)^j x^{2j}\\
&=&\binom{\frac{n+k}{2}}{k}(-1)^{\frac{n-k}{2}}\frac{1+(-1)^{n-k}}{2}+r\binom{\frac{n+k-2}{2}}{k}(-1)^{\frac{n-k}{2}}\frac{1+(-1)^{n-k}}{2}+r.0^{n+k}.
\end{eqnarray*}
\end{proof}
\begin{corollary}
We have
$$P_{n,k}=\left(\binom{\frac{n+k}{2}}{k}+r \binom{\frac{n+k-2}{2}}{k}\right)(-1)^{\frac{n-k}{2}}\frac{1+(-1)^{n-k}}{2}-r.0^{n+k},$$ and
$$P_{n,k}=\binom{\frac{n+k}{2}}{k}\left(1-\frac{r(n-k)}{n+k+0^{n+k}}\right)(-1)^{\frac{n-k}{2}}\frac{1+(-1)^{n-k}}{2}.$$
\end{corollary}
Thus we have
$$P_n(x;r)=\sum_{k=0}^n P_{n,k} x^k.$$
Two special cases have other well-known expressions.
We have
$$P_n(x;0)=U_n(x/2)=\sum_{k=0}^{\lfloor \frac{n}{2} \rfloor}\binom{n-k}{k}(-1)^k x^{n-2k},$$ and
$$P_n(x;3)=B_n(x)=\sum_{k=0}^{\lfloor \frac{n}{2} \rfloor}\binom{n-k}{k}\frac{n-4k}{n-k}(-1)^k x^{n-2k}.$$
In general, we have
$$P_n(x;r)=\sum_{k=0}^{\lfloor \frac{n}{2} \rfloor}\binom{n-k}{k}\frac{n-(r+1)k}{n-k}(-1)^k x^{n-2k}.$$

We now turn to look at the central terms $P_{2n,n}$ and $P_{2n,n+1}$ of the coefficient matrix $P(r)$.
We have
\begin{proposition} $$P_{2n,n}(r)=0^n-(r-3)(-1)^{\frac{n}{2}}\binom{\frac{3n}{2}-1}{\frac{n}{2}-1}\frac{1+(-1)^n}{2}.$$ Alternatively, $$P_{2n,n}(r)=(r-2)0^n -(r-3)(-1)^{\frac{n}{2}}\binom{\frac{3n}{2}-1}{n}\frac{1+(-1)^n}{2}.$$
\end{proposition}
\begin{corollary}
For the Boubaker polynomials, we have $$B_{2n,n}=P_{2n,n}(3)=0^n.$$
\end{corollary}
\begin{proof} In order to prove this result, we will calculate the generating function of the term $P_{2n,n}$.
For this, we use the result \cite{Central}: Let $d_{2n,n}$ be the central coefficient sequence of the Riordan array $(d(t), h(t))$. Then we have
$$d_{2n,n}=[t^n] \frac{d(v(t))}{\frac{f}{t}(v(t))} \frac{d}{dt} v(t),$$ where
$$v(t)=\overline{\left(\frac{t^2}{h(t)}\right)}.$$
In our case (using the ``dummy'' parameter $x$), we have
$$h(x)=\frac{x}{1+x^2} \Rightarrow v(x)=\overline{x(1+x^2)}=\frac{432^{\frac{1}{6}}}{6}\left((\sqrt{27x^2+4}+3\sqrt{3}x)^{\frac{1}{3}}-(\sqrt{27x^2+4}-3\sqrt{3}x)^{\frac{1}{3}}\right).$$
We obtain that the generating function of $P_{2n,n}$ is given by
\begin{scriptsize}
$$\left(\frac{r}{3}+\frac{(\sqrt{27x^2+4}+3\sqrt{3}x)^{\frac{1}{3}}(r(\sqrt{27x^2+4}-3\sqrt{3}x)^{\frac{2}{3}}-(r-3)2^{\frac{2}{3}})+2^{\frac{2}{3}}(3-2r)(\sqrt{27x^2+4}-3\sqrt{3})^{\frac{1}{3}}}{6 \sqrt{27x^2+4}}\right).$$\end{scriptsize}
We note that when $r=3$, this reduces to $1$.
\end{proof}

The sequence $P_{2n,n}(r)$ expands to give
$$1, 0, r - 3, 0, 5(3 - r), 0, 28(r - 3), 0, 165(3 - r), 0, 1001(r - 3), 0, 6188(3 - r), 0, \ldots,$$ or
$$1, 0, r - 3, 0, 15 - 5r, 0, 28r - 84, 0, 495 - 165r, 0, 1001r - 3003, 0, 18564 - 6188r, 0, \ldots.$$
We recognise in the numbers $$1,0,-3,0,15,0,-84,0,495,0,-3003,\ldots$$ an aerated version of the numbers
$$\binom{3n}{n}(-1)^n.$$
Similarly, the numbers
$$0, 0, 1, 0, -5, 0, 28, 0, -165, 0, 1001, 0, -6188, 0, 38760, 0, -245157, 0, 1562275, 0,\ldots$$ are an aerated version of the numbers $$\binom{3n-1}{n-1}(-1)^n.$$
The Hankel transform $h_n$ of $d_{2n,n}(r)$ is such that the sequence
$\frac{h_n}{(r-3)^n}$ begins
$$1, 1, -r - 2, - 3(r + 2), 3(10r + 11), 26(10r + 11), - 52(108r + 85), - 1292(108r + 85),\ldots.$$ This last sequence is the sum of the sequence
$$1, 1, -2, -6, 33, 286, -4420, -109820, 4799134, 340879665,\ldots,$$ and the sequence
$$r \{0, 0, -1, -3, 30, 260, -5616, -139536, 7683524, 545756190, -80623845225,\ldots\}.$$
We recognise in the former sequence $(-1)^{\binom{n}{2}}$ times the number of alternating sign $(2n+3) \times (2n+3)$ matrices symmetric with respect to both horizontal and vertical axes (\seqnum{A005161}).

It is interesting to calculate the Hankel transform $H_n$ of the unaerated sequence
$$-(r-3)(-1)^n \binom{3n-1}{n-1}+0^n$$ which begins
$$1, r - 3, 5(3 - r), 28(r - 3), 165(3 - r), 1001(r - 3), 6188(3 - r),\ldots,$$ or
$$1, r - 3, 15 - 5r, 28r - 84, 495 - 165r, 1001r - 3003, 18564 - 6188r,\ldots,$$ or equivalently
$$\{1, -3, 15, -84, 495, -3003, 18564,\ldots\}+r \{0, 1, -5, 28, -165, 1001, -6188,\ldots\}.$$
We find that $H_n/(r-3)^n$ is equal to
\begin{equation*}\begin{split}\{1, -2, 11, -170, 7429, -920460, 323801820, \ldots\}+\\\quad\quad\quad r\{0, 1, -10, 216, -11894, 1757085, -712169820, \ldots\}.\end{split}\end{equation*}
The first sequence is $(-1)^n$ times the number of cyclically symmetric transpose complement plane partitions in a $(2n+2) \times (2n+2) \times (2n+2)$ box (\seqnum{A051255}).

Turning now to the term $P_{2n,n+1}(r)$, we have
\begin{eqnarray*} P_{2n,n+1}&=&\left(\binom{\frac{3n+1}{2}}{n+1}-r \binom{\frac{3n-1}{2}}{n+1}\right)(-1)^{\frac{n-1}{2}}\frac{1-(-1)^n}{2}\\
&=&\binom{\frac{3n+1}{2}}{n+1}\left(\frac{1+r-n(r-3)}{3n+1}\right)(-1)^{\frac{n-1}{2}}\frac{1-(-1)^n}{2}.\end{eqnarray*} This sequence begins
\begin{equation*}\begin{split}\{0 ,1, 0, -5, 0, 28, 0, -165, 0, 1001, 0, -6188, 0, \ldots\}+\\\quad\quad\quad r\{0, 0, 0, 1, 0, -7, 0, 45, 0, -286, 0, 1820, 0, \ldots\}.\end{split}\end{equation*}
In the special case $r=3$ of the Boubaker polynomials, we obtain that $P_{2n,n+1}(3)$ begins
$$0, 1, 0, -2, 0, 7, 0, -30, 0, 143, 0,\ldots$$ wherein we recognise a signed aerated version of $\frac{1}{n+1} \binom{3n+1}{n}$ which begins $1,2,7,30,143,\ldots$, \seqnum{A006013}. In this case, the Hankel transform of $P_{2n,n+1}$ begins
$$0, -1, 0, 9, 0, -676, 0, 417316, 0, -2105433225, 0,\ldots.$$ We note that the sequence beginning
$1,1,9,676,417316, \ldots,$ is the expansion of the generating function $A_{UO}^1(8n)=A_V(2n+1)^2$, \seqnum{A059492}. This sequence is the square of the sequence $A_V(2n+1)$ beginning $1,1,3,26,646,\ldots$.
We note that the generating function of $P_{2n,n+1}(r)$ can be shown, by the methods of \cite{Central}, to be
\begin{scriptsize}
$$108^{\frac{1}{6}}\frac{ (\sqrt{27x^2 + 4} + 3\sqrt{3}x)^{\frac{2}{3}}(r(\sqrt{27x^2 + 4} + 3\sqrt{3}x)^{\frac{2}{3}} - 2^{\frac{2}{3}}(2r - 3)) - (‹(27·x^2 + 4) - 3\sqrt{3}x)^{\frac{2}{3}}(r(\sqrt{27x^2 + 4} - 3\sqrt{3}x)^{\frac{2}{3}} - 2^{\frac{2}{3}}(2·r - 3))}{6((\sqrt{27x^2 + 4} - 3\sqrt{3}x)^{\frac{2}{3}} + (\sqrt{27x^2 + 4} + 3\sqrt{3}x)^{\frac{2}{3}} + 2^{\frac{2}{3}})\sqrt{27x^2 + 4}}.$$
\end{scriptsize}
\section{The moment matrix $M(r)$}
In this section we study aspects of the moment matrix $M(r)$ given by the inverse of the polynomial coefficient matrix. Thus we have
$$M(r)=\left(\frac{1+rx^2}{1+x^2}, \frac{x}{1+x^2}\right)^{-1}=\left(\frac{\sqrt{1-4x^2}(r-1)+r+1}{2(r+x^2(r-1)^2)}, \frac{1-\sqrt{1-4x^2}}{2x}\right).$$
The first column of this matrix, with generating function
$$\frac{\sqrt{1-4x^2}(r-1)+r+1}{2(r+x^2(r-1)^2)},$$ is the moment sequence of the family of orthogonal polynomials $P_n(x;r)$.
We see that this is an aerated sequence $\mu_n(r)$, beginning
\begin{equation*}\begin{split}1, 0, 1 - r, 0, r^2 - 3r + 2, 0, - r^3 + 5r^2 - 9r + 5, 0, r^4 - 7r^3 + 20r^2 - 28r + 14, 0,\\ - r^5 + 9r^4 - 35r^3 + 75r^2 - 90r + 42, 0,\ldots.\end{split}\end{equation*}
We recognise in the un-aerated sequence
$$1,1-r,r^2-3r+2,- r^3 + 5r^2 - 9r + 5,\ldots,$$ the polynomial sequence with generating function given by
$$(c(x), 1-c(x)) \cdot \frac{1}{1-rx},$$
with general term
$$\sum_{k=0}^n \frac{2k+1}{n+k+1} \binom{2n}{n-k}(-1)^k r^k.$$
Hence we have
$$\mu_n(r)=\frac{1+(-1)^n}{2} \sum_{k=0}^{\frac{n}{2}} \frac{2k+1}{\frac{n}{2}+k+1}\binom{n}{\frac{n}{2}-k}(-1)^k r^k.$$
Noting that the production matrix of $M(r)$ is given by the tri-diagonal matrix
$$\left(
\begin{array}{ccccccc}
 0 & 1 & 0 & 0 & 0 & 0 & 0 \\
 1-r & 0 & 1 & 0 & 0 & 0 & 0 \\
 0 & 1 & 0 & 1 & 0 & 0 & 0 \\
 0 & 0 & 1 & 0 & 1 & 0 & 0 \\
 0 & 0 & 0 & 1 & 0 & 1 & 0 \\
 0 & 0 & 0 & 0 & 1 & 0 & 1 \\
 0 & 0 & 0 & 0 & 0 & 1 & 0 \\
\end{array}\right),$$ we deduce that the generating function $\frac{\sqrt{1-4x^2}(r+1)+r-1}{2(r-x^2(r+1)^2)}$ of the moment sequence $\mu_n(r)$ has a continued fraction expression as
$$\cfrac{1}{1-
\cfrac{(1-r)x^2}{1-
\cfrac{x^2}{1-
\cfrac{x^2}{1-\cdots}}}}.$$
We immediately deduce that
\begin{proposition} The Hankel transform $h_n(r)$ of the moment sequence $\mu_n(r)$ of the family of orthogonal polynomials $P_n(x;r)$ is given by
$$h_n(r)=(1-r)^n.$$
\end{proposition}
An interesting feature of this sequence is that the Hankel transform of the un-aerated sequence is also equal to $(1-r)^n$. This follows since the un-aerated sequence has a generating function given by the Stieltjes continued fraction
$$\cfrac{1}{1-
\cfrac{(1-r)x}{1-
\cfrac{x}{1-
\cfrac{x}{1-\cdots}}}}.$$
Starting from the generating function $\frac{\sqrt{1-4x^2}(r-1)+r+1}{2(r+x^2(r-1)^2)}$ of the moment sequence $\mu_r(r)$, and invoking the Stieltjes-Perron theorem \cite{Gautschi, Henrici, Wall}, we arrive at the following result.
\begin{proposition} We have the following integral representation of the moment sequence $\mu_n(r)$:
$$\mu_n(r)=\frac{-1}{\pi} \int_{2}^2 x^n \frac{\sqrt{4-x^2}(r-1)}{2(rx^2+(r-1)^2)}\,dx +\frac{r+1}{2r}\left(-\frac{r-1}{\sqrt{-r}}\right)^n + \frac{r+1}{2r}\left(\frac{r-1}{\sqrt{-r}}\right)^n.$$
\end{proposition}

We now turn to the row sums $s_n$ of the moment matrix. By the theory of Riordan arrays, these will have their generating function given by
$$\frac{\sqrt{1-4x^2}(r-1)+r+1}{2(r+x^2(r-1)^2)} \cdot \frac{1}{1-\frac{1-\sqrt{1-4x^2}}{2x}}=\frac{\sqrt{1-4x^2}(x(r-1)-r)+2x^2(r-1)+x(r+1)-r}{2(x-1)(x^2(r-1)^2+r)}.$$
\begin{proposition} We have
$$s_n=\sum_{k=0}^{\lfloor \frac{n}{2} \rfloor} \binom{n}{\lfloor \frac{n-2k}{2} \rfloor}(-1)^k r^k.$$
\end{proposition}
\begin{proof}
The matrix with $(n,k)$-th element $\binom{n+k}{\lfloor \frac{n-k}{2} \rfloor}$ is the Riordan array
$$\left(\frac{1+xc(x^2)}{\sqrt{1-4x^2}}, \frac{c(x^2)-1}{x}\right).$$ Our assertion is that the above generating function for $s_n$ is equal to
$$\left(\frac{1+xc(x^2)}{\sqrt{1-4x^2}}, 1-c(x^2)\right)\cdot \frac{1}{1-rx},$$ which can be verified by direct calculation.
\end{proof}
We next look at the Hankel transform $H_n(r)$ of the row sum sequence. We find that the sequence $H_n(r)$ begins
$$1, 1 - r, 1 - 2r, r^3 + r^2 - 3r + 1, - r^4 + 2r^3 + 3r^2 - 4r + 1, - 4r^4 + 2r^3 + 6r^2 - 5r + 1, \ldots$$ with coefficient array that begins
$$\left(
\begin{array}{ccccccc}
 1 & 0 & 0 & 0 & 0 & 0 & 0 \\
 1 & -1 & 0 & 0 & 0 & 0 & 0 \\
 1 & -2 & 0 & 0 & 0 & 0 & 0 \\
 1 & -3 & 1 & 1 & 0 & 0 & 0 \\
 1 & -4 & 3 & 2 & -1 & 0 & 0 \\
 1 & -5 & 6 & 2 & -4 & 0 & 0 \\
 1 & -6 & 10 & 0 & -9 & 2 & 1 \\
\end{array}
\right).$$ Now the reversal of this matrix, which begins
$$\left(
\begin{array}{ccccccc}
 1 & 0 & 0 & 0 & 0 & 0 & 0 \\
 -1 & 1 & 0 & 0 & 0 & 0 & 0 \\
 0 & -2 & 1 & 0 & 0 & 0 & 0 \\
 1 & 1 & -3 & 1 & 0 & 0 & 0 \\
 -1 & 2 & 3 & -4 & 1 & 0 & 0 \\
 0 & -4 & 2 & 6 & -5 & 1 & 0 \\
 1 & 2 & -9 & 0 & 10 & -6 & 1 \\
\end{array}
\right),$$ is the Riordan array
$$\left(\frac{1}{1+x+x^2}, \frac{x}{1+x+x^2}\right).$$
We then have
$$\left(\frac{1}{1+x+x^2}, \frac{x}{1+x+x^2}\right)=\left(\frac{1}{1+x^2}, \frac{x}{1+x^2}\right)\cdot \left(\frac{1}{1-x}, \frac{x}{1-x}\right),$$ where the $\left(\frac{1}{1+x^2}, \frac{x}{1+x^2}\right)$ is the coefficient array of the modified Chebyshev polynomials $U_n(\frac{x}{2})$, and $\left(\frac{1}{1-x}, \frac{x}{1-x}\right)$ is the binomial matrix $\left(\binom{n}{k}\right)$. We deduce the following.
\begin{proposition} The Hankel transform sequence $H_n(r)$ of the row sum sequence $s_n(r)$ of the moment matrix of the family of restricted Chebyshev-Boubaker polynomials $P_n(r)$ is given by
$$ H_n(r)=r^n U_n\left(\frac{\frac{1}{r}-1}{2}\right).$$
\end{proposition}
\begin{corollary}  The row sum sequence $s_n(r)$ of the moment matrix of the family of restricted Chebyshev-Boubaker polynomials $P_n(r)$ are the moments of the family of orthogonal polynomials $Q_n(x;r)$ that satisfy the three term recurrence
$$Q_n(x;r)=(x-\alpha_n ) Q_{n-1}(x;r)-\beta_n Q_{n-2}(x;r),$$ where
$$\alpha_n = \frac{r^{2n+1}}{H_{n-1}(r)H_n(r)},$$ and
$$\beta_n = \frac{H_{n-1}(r)H_{n+1}(r)}{H_n(r)^2},$$
with $$ H_n(r)=r^n U_n\left(\frac{\frac{1}{r}-1}{2}\right).$$
\end{corollary}
We note that the sequence $s_n(r)$ has the moment representation

\begin{equation*} \begin{split}s_n(r)=-\frac{1}{\pi}\int_{-2}^2 x^n \frac{\sqrt{4-x^2}(rx-(r-1))}{2(2-x)(rx^2+(r-1)^2)}\,dx
+\left(\frac{1}{2}+\frac{\sqrt{r}}{2r}i\right)\left(-\frac{r-1}{\sqrt{r}}i\right)^n\\\quad\quad+
\left(\frac{1}{2}-\frac{\sqrt{r}}{2r}i\right)\left(\frac{r-1}{\sqrt{r}}i\right)^n.\end{split}\end{equation*}

\section{Conclusions} The theory of Riordan arrays provides a useful context within which to discuss the family of restricted Chebyshev-Boubaker orthogonal polynomials. These polynomials give us examples of polynomial sequences with interesting properties, most notably linked to their Hankel transforms. These sequences appear as central coefficients in the coefficient array of the family of orthogonal polynomials under study, and as sequences of moments and generalized moments (row sums of the moment matrix) of the same families.

The author declares that there are no conflicts of interest regarding the
publication of this article.

\bigskip
\hrule
\bigskip
\noindent 2010 {\it Mathematics Subject Classification}: Primary
15B36; Secondary 33C45, 11B83, 11C20, 05A15.
\noindent \emph{Keywords:} Riordan array, orthogonal polynomials, Chebyshev polynomials, Boubaker polynomials, generating functions.


\begin{thebibliography}{9}

\bibitem{Agida}
M. Agida and A.~S. Kumar, A Boubaker polynomials expansion scheme solution to random
Love's equation in the case of a rational kernel, \emph{Elec. J. Theoretical Phys.}, \textbf{7} (2010), 319--326.

\bibitem{Arregui}
J. L. Arregui,
 \\{Tangent and Bernoulli numbers related to Motzkin and Catalan numbers by means of numerical triangles}, (2001) \href{http://arxiv.org/abs/math/0109108}{arXiv:math/0109108}.

\bibitem{NMR}
O.B. Awojoyogbe and K, Boubaker, A solution to Bloch NMR flow equations for the analysis of hemodynamic functions of blood flow system using $m$-Boubaker polynomials
\emph{Current Applied Physics}, \textbf{9} (2009), 278--283

\bibitem{Meixner} P. Barry and A. Hennessy, Meixner-type results for Riordan arrays and
 associated integer sequences, \emph{J. Integer Seq.}, \textbf{13} (2010), \href{https://cs.uwaterloo.ca/journals/JIS/VOL13/Barry5/barry96s.pdf} {Article 10.9.4}.

\bibitem{CB} P. Barry, On the connection coefficients of the Chebyshev-Boubaker polynomials, \emph{The Scientific World Journal}, \textbf{2013}, Article ID 657806, 10 pages.

\bibitem{Central} P. Barry, On the central coefficients of Riordan matrices, \emph{J. Integer Seq.}, \textbf{16} (2013), \href{https://cs.uwaterloo.ca/journals/JIS/VOL16/Barry1/barry242.pdf} {Article 13.5.1}.

\bibitem{Boubaker} K. Boubaker and L. Zhang, Fermat-linked relations for the Boubaker polynomial sequences via Riordan matrices analysis, \emph{J. Assoc. Arab Univ. Basic Appl. Sci.}, \textbf{12} (2012), 74--78.

\bibitem{BPES}
K. Boubaker, Analytical initial-guess-free solution to Kepler’s transcendental equation using Boubaker Polynomials Expansion Scheme BPES
\emph{Apeiron: Studies in Infinite Nature}, \textbf{17} (2010), 1--12.

\bibitem{Boubaker07}
K. Boubaker, A. Chaouachi, M. Amlouk, and H. Bouzouita, Enhancement of pyrolysis spray dispersal performance using thermal time-response to precursor uniform deposition, \emph{Eur. Phys. J. AP}, \textbf{37} (2007),  105--109.

\bibitem{Cheon} G-S. Cheon, H. Kim, and L. W. Shapiro, Riordan group involutions, \emph{Linear Algebra Appl.},
    \textbf{428}
    (2008), 941--952.

\bibitem{Chihara}
T. S. Chihara,  {\it An Introduction to Orthogonal Polynomials},
Gordon and Breach, New York.

\bibitem{Dada}
O.~M. Dada, O.~B. Awojoyogbe, M. Agida, and K. Boubaker, Variable separation and Boubaker polynomial
expansion scheme for solving the neutron transport equation, \emph{Physics International}, \textbf{2} (2011), 25--30.


\bibitem{ProdMat_0} E. Deutsch, L. Ferrari, and S. Rinaldi, Production matrices, \emph{Adv. in Appl. Math.},
    \textbf{34} (2005), 101--122.

\bibitem{ProdMat} E. Deutsch, L. Ferrari, and S. Rinaldi,
Production matrices and Riordan arrays, \emph{Ann. Comb.}, \textbf{13} (2009), 65--85.

\bibitem{Gautschi}
W. Gautschi, {\it Orthogonal Polynomials: Computation and
Approximation}, Clarendon Press, Oxford.

\bibitem{He}
Tian-Xiao He, R. Sprugnoli, Sequence characterization of Riordan arrays, \emph{Discrete Math.} \textbf{2009} (2009), 3962--3974.

\bibitem{Henrici} P. Henrici, \emph{Applied and Computational Complex Analysis, vol. $3$}, John Wiley \& Sons, 1993.

\bibitem{Kratt} C. Krattenthaler, Advanced Determinant
    Calculus, available electronically at
    \texttt{http://arxiv.org/PS\_cache/math/pdf/9902/9902004.pdf},
    2010.

\bibitem{Labadiah1}
H. Labadiah and K. Boubaker, A Sturm-Liouville shaped characteristic differential equation as a guide to establish a quasi-polynomial expression to the Boubaker polynomials, \emph{J. of Diff. Eq. and C. Proc.}, (2007) 117-133.

\bibitem{Labadiah2}
H. Labadiah, M. Dada, B. Awojoyogbe, B. Mahmoud, and A. Bannour, Establishment of an ordinary generating function and a Christoffel-Darboux type first-order differential equation for the heat equation related Boubaker-Turki polynomials,
\emph{J. of Diff. Eq. and C. Proc.}, (2008), 51--66.

\bibitem{Layman} J. W. Layman, The Hankel transform and some of its properties, \emph{J. Integer Seq.}, \textbf{4} (2001),
\href{https://www.cs.uwaterloo.ca/journals/JIS/VOL4/LAYMAN/hankel.html} {Article 01.1.5}.

\bibitem{Mason} J. C. Mason and D. C. Handscomb, \emph{Chebyshev Polynomials}, Chapman and Hall/CRC, 2002.

\bibitem{MC} D. Merlini, R. Sprugnoli and M.~C.
    Verri, The Method of Coefficients, \emph{Amer. Math. Monthly}, \textbf{114} (2007), 40--57.

\bibitem{Milanovic} G. V. Milanovic and D. Joksimovic, Properties of Boubaker polynomials and an application to Love's integeral equation, \emph{Appl. Math. Comput.}, \textbf{224} (2013), 74--87.

\bibitem{P_W} P. Peart, W.-J. Woan, Generating functions via Hankel and Stieltjes matrices, \emph{J.
Integer Seq.}, \textbf{3} (2000), \href{http://www.cs.uwaterloo.ca/journals/JIS/VOL3/PEART/peart1.html} {Article 00.2.1}.

\bibitem{SGWW} L.~W.~Shapiro, S. Getu, W.-J. Woan, and L.C. Woodson,
The Riordan group, \emph{Discr. Appl. Math.} \textbf{34} (1991), 229--239.

\bibitem{Survey} L. Shapiro, A survey of the Riordan group, available electronically at
\href{http://www.combinatorics.cn/activities/Riordan\%20Group.pdf}, Center for Combinatorics, Nankai University, 2005.

\bibitem{Shapiro_bij} L.~W.~Shapiro, Bijections and the Riordan group, \emph{Theoret. Comput. Sci.} \textbf{307} (2003), 403--413.

\bibitem{Slama} S. Slama, J. Bessrour, K. Boubaker, and M. Bouhafs, A dynamical model for investigation of a $3$ point maximal spatial evolution during resistance spot welding using Boubaker polynomials,
\emph{Eur. Phys. J. AP}, \textbf{44} (2008), 317--322.

\bibitem{SL1} N. J. A.~Sloane, \emph{The
On-Line Encyclopedia of Integer Sequences}. Published electronically
at \texttt{http://oeis.org}, 2016.

\bibitem{SL2} N. J. A.~Sloane, The On-Line Encyclopedia of Integer
Sequences, \emph{Notices Amer. Math. Soc.}, \textbf{50} (2003),  912--915.


\bibitem{Spru} R. Sprugnoli, Riordan arrays and combinatorial sums,
\emph{Discrete Math.} \textbf{132} (1994), 267--290.

\bibitem{Szego} G. Szeg\"o, \emph{Orthogonal Polynomials}, 4e,
     Providence, RI, Amer. Math. Soc., 1975.

\bibitem{Wall} H.~S. Wall, \emph{Analytic Theory of
    Continued Fractions}, AMS Chelsea Publishing.

\bibitem{Zhao} T. G. Zhao, L. Naing and W. X. Yue, Some new features of the Boubaker polynomials expansion scheme BPES, \emph{J. Assoc. Arab Univ. Basic Appl. Sci.}, \textbf{12} (2012) 74--78.

\end{thebibliography}
\end{document}